\documentclass{amsart}

\usepackage{amsmath, amsthm, amssymb, amstext, amsfonts}
\usepackage{enumerate}
\usepackage{graphicx}
\usepackage[applemac]{inputenc}
\usepackage[T1]{fontenc} 

\theoremstyle{plain}

\newtheorem{thm}{Theorem}[section]
\newtheorem{lem}[thm]{Lemma}

\newtheorem{prop}[thm]{Proposition}

\newtheorem*{prob*}{Problem}

\theoremstyle{definition}

\newtheorem{rem}[thm]{Remark}

\newcommand{\N}{\ensuremath{\mathbb{N}}}

\newcommand{\sm}{\ensuremath{\smallsetminus}}

\newcommand{\diam}{\textnormal{diam}}

\newcommand{\inv}{\ensuremath{^{-1}}}

\newcommand{\rand}{\partial}
\newcommand{\Aut}{\textnormal{Aut}}
\newcommand{\asdim}{{\rm asdim}}

\newcommand{\es}{\ensuremath{\emptyset}}

\newcommand{\sub}{\subseteq}

\newcommand{\Int}{\textnormal{Int}}
\newcommand{\mult}{\textnormal{mult}}

\def\ta{tree amalgamation}
\def\qt{quasi-tran\-si\-tive}
\def\qi{quasi-iso\-metric}
\def\qiy{quasi-iso\-me\-try}

\newcommand{\comment}[1]{}

\newcommand{\real}{{\mathbb R}}

\newcommand{\fwd}{\overset{\scriptscriptstyle\rightarrow}}
\newcommand{\bwd}{\overset{\scriptscriptstyle\leftarrow}}

\newcommand{\Scal}{\ensuremath{\mathcal{S}}}
\newcommand{\Ucal}{\ensuremath{\mathcal{U}}}
\newcommand{\Vcal}{\ensuremath{\mathcal{V}}}
\newcommand{\Wcal}{\ensuremath{\mathcal{W}}}

\def\?#1{\vadjust{\vbox to 0pt{\vss\vskip-8pt\leftline{%
     \llap{\hbox{\vbox{\pretolerance=-1
     \doublehyphendemerits=0\finalhyphendemerits=0
     \hsize16truemm\tolerance=10000\small
     \lineskip=0pt\lineskiplimit=0pt
     \rightskip=0pt plus16truemm\baselineskip8pt\noindent
     \hskip0pt        
     #1\endgraf}\hskip7truemm}}}\vss}}}

\newenvironment{txteq*}
  {
    \begin{equation*}
    \begin{minipage}[c]{0.85\textwidth} 
    \em                                
  }
  {\end{minipage}\end{equation*}\ignorespacesafterend}

\begin{document}

\title[Asymptotic dimension of multi-ended quasi-transitive graphs]{Asymptotic dimension\\of multi-ended quasi-transitive graphs}
\author{Matthias Hamann}
\thanks{Supported by the Heisenberg-Programme of the Deutsche Forschungsgemeinschaft (DFG Grant HA 8257/1-1).}
\address{Matthias Hamann, Mathematics Institute, University of Warwick, Coventry, UK}
\date{}

\begin{abstract}
We prove the existence of an upper bound on the asymptotic dimension of \ta s of locally finite \qt\ connected graphs.
This generalises a result of Dranishnikov for free products with amalgamation and a result of Tselekidis for HNN-extensions of groups to \ta s of graphs.
As a corollary, we obtain an upper bound on the asymptotic dimension of a multi-ended \qt\ locally finite graph based on any of their factorisations.
\end{abstract}

\maketitle

\section{Introduction}\label{sec_Intro}

Asymptotic dimension of metric spaces was introduced by Gromov~\cite{G-AsymptoticInv}.
It is a \qiy\ invariant and hence an invariant of finitely generated groups.
Thus, it is interesting to see how the asymptotic dimension behaves with respect to free products with amalgamations and HNN-extensions.
Bell and Dranishnikov~\cite{BD-AsdimGroupOnTrees} proved for these products that the asymptotic dimension is finite provided that the asymptotic dimension of the factors is finite.
The best upper bound for the free product with amalgamation $A\ast_CB$ of finitely generated groups $A$ and~$B$ is given by Dranishnikov~\cite{D-AsdimAmalgamCoxeter}:
\[
\asdim(A\ast_CB)\leq\max\{\asdim(A),\asdim(B),\asdim(C)+1\}.
\]
Recently, Tselekidis~\cite{T-AsdimHNN} obtained a similar upper bound for the HNN-extension $A\ast_C$ of a finitely generated group~$A$:
\[
\asdim(A\ast_C)\leq\max\{\asdim(A),\asdim(C)+1\}.
\]

We generalise these two results in our main theorem.

\begin{thm}\label{thm_main}
	Let $G_1,G_2$ be locally finite connected graphs and let $\Gamma_1,\Gamma_2$ be groups acting \qt ly on $G_1,G_2$, respectively.
	Let $G=G_1\ast G_2$ be the \ta\ of finite identification respecting the actions of $\Gamma_1$ and $\Gamma_2$.
	Then
	\[
	\asdim(G)\leq \max\{\asdim(G_1),\asdim(G_2),\asdim(C)+1\},
	\]
	where $C$ is an arbitrary adhesion set of $G_1\ast G_2$.
\end{thm}

We note that neither in the results of Dranishnikov and of Tselekidis the group $C$ has to be finite nor in our result the adhesion set $C$ has to be finite.

The bound in Theorem~\ref{thm_main} is sharp as some \ta s of finite graphs are \qi\ to trees by~\cite[Theorem 7.4]{HLMR} and thus have asymptotic dimension~$1$ but finite graphs have asymptotic dimension~$0$.

Let $G$ be a locally finite \qt\ connected graph with more than one end.
A tuple $(G_1,\ldots,G_n)$ of locally finite \qt\ connected graphs is a \emph{factorisation} of~$G$ if $G$ is obtained by iterated non-trivial \ta s of all $G_i$ of finite adhesion and finite identification respecting the group actions.

By~\cite{HLMR}, all multi-ended \qt\ locally finite connected graphs have a \emph{non-trivial} factorisation, i.\,e.\ a factorisation with more than one factor.
But there are examples of Dunwoody~\cite{D-AnInaccessibleGroup,D-AnInaccessibleGraph} that show that not every such graph has a \emph{terminal} factorisation, i.\,e.\ a factorisation all of whose factors have at most one end.

The following result is an immediate corollary of Theorem~\ref{thm_main}.

\begin{thm}\label{thm_mainCor}
	Let $(G_1,\ldots,G_n)$ be a factorisation of a locally finite \qt\ connected graph $G$.
	Then $\asdim(G)\leq\max\{1, \asdim(G_i)\mid 1\leq i\leq n\}$.\qed
\end{thm}

\section{Tree amalgamations}

In this section, we will define all notations and state all results that we need in the context of \ta s.

A tree is \emph{$(p_1,p_2)$-semiregular} if for the canonical bipartition $\{V_1,V_2\}$ of its vertex set all vertices in~$V_i$ have degree~$p_i$ for $i=1,2$.

Let~$I_1$ and~$I_2$ be disjoint sets and let $G_1$ and~$G_2$ be graphs.
Let $(S_k^i)_{k \in I_i}$ be families of subsets of~$V(G_i)$ for $i=1,2$ such that all sets~$S_k^i$ have the same cardinality.
For all $k \in I_1$ and $\ell \in I_2$, let $\phi_{k \ell} \colon S_k^1 \rightarrow S_\ell^2$ be a bijection.
Set $\phi_{\ell k} = \phi_{k\ell}^{-1}$. 
The maps $\phi_{k\ell}$ and $\phi_{\ell k}$ are the \emph{bonding maps}.

Let~$T$ be a \emph{$(|I_1|,|I_2|)$-semiregular} tree with canonical bipartition $\{V_1,V_2\}$ such that the vertices in~$V_i$ have degree~$|I_i|$.
Let $D(T)$ be the set obtained from the edge set of~$T$ by replacing every $xy\in E(T)$ by two directed edges $\fwd {xy}$ and $\fwd {yx}$. 
For a directed edge $\fwd{e} = \fwd{xy} \in D(T)$, we denote by $\bwd e$ the edge with the reversed orientation, i.\,e.\ $\bwd{e}=\fwd{yx}$.
Let $f\colon D(T) \to I_1 \cup I_2$ be a labelling such that for every $i\in\{1,2\}$ and every $v \in V_i$ each $k\in I_i$ occurs in the set of labels of edges starting at~$v$ precisely once.

For every $i\in\{1,2\}$ and every $v\in V_i$, let $G_v$ be a copy of~$G_i$. 
Denote by~${S^v_k}$ the corresponding copies of~$S_k^i$ in~$V(G_v)$. 
Let $G_1+G_2$ be the graph obtained from the disjoint union of all graphs $G_v$ with $v\in V(T)$ by adding new edges as follows:
for every edge $\fwd e = \fwd{uv}$ with $f(\fwd e) = k$ and $f(\bwd e) = \ell$ we add an edge between each $x\in S^{u}_k$ and $\phi_{k\ell}(x)\in S^v_\ell$.
Note that this does not depend on the orientation we pick for $e$, since $\phi_{\ell k} = \phi_{k\ell}^{-1}$.
If we contract all edges of $G_1+G_2$ that lie outside of the copies $G_v$ we obtain the \emph{tree amalgamation} $G_1\ast_T G_2$ of the graphs~$G_1$ and~$G_2$ over the \emph{connecting tree}~$T$.
If $T$ is clear from the context, we just write $G_1\ast G_2$.
Let $\pi\colon V(G_1+G_2)\to V(G_1\ast G_2)$ be the canonical map that maps each vertex of $G_1+G_2$ to the vertex obtained from it after all the contractions.

The sets $S_k^i$ and their images under~$\pi$ in $G_1\ast G_2$ are the \emph{adhesion sets} of the \ta.
A \ta\ has \emph{finite adhesion} if all its adhesion sets are finite. 
We call a \ta\ $G_1\ast_T G_2$ \emph{trivial} if for some $v\in V(T)$ the restriction of~$\pi$ to $V(G_v)$ is a bijection.

For a vertex $x\in V(G_1\ast_T G_2)$ let $T_x$ be the maximal subtree of~$T$ such that every node of~$T_x$ contains a vertex $y$ with $\pi(y)=x$.
The \emph{identification size} of~$x$ is the cardinality of $V(T_x)$.
The \ta\ has \emph{finite identification} if all identification sizes of its vertices are bounded.

Let $G$ and~$H$ be graphs.
A map $f\colon V(G)\to V(H)$ is a \emph{\qiy} if there are constants $\gamma\geq 1$ and $c\geq 0$ such that
\[
\gamma\inv d_G(x,y)-c\leq d_H(f(x),f(y))\leq\gamma d_G(x,y)+c
\]
for all $x,y\in V(G)$.

\begin{rem}\label{rem_QI+to*}
Let $G$ and $H$ be graphs.
It is obvious from the construction that the map $\pi\colon V(G+H)\to V(G\ast H)$ is a \qiy.
\end{rem}

So far, the \ta\ is independent from any group action.
In the following, we describe some conditions on \ta s that ensure that \ta s of \qt\ graphs are again \qt, see \cite[Lemma~5.3]{HLMR}.

For $i=1,2$, let $\Gamma_i$ be a group acting on~$G_i$.
Let $i\in\{1,2\}$.
The \ta\ \emph{respects $\gamma\in\Gamma_i$} if there is a permutation $\pi$ of~$I_i$ such that for every $k \in I_i$ there exist $\ell \in I_j$ and $\tau$ in the setwise stabiliser of~$S_\ell$ in~$\Gamma_j$ such that 
\[
\phi_{k\ell} = \tau \circ \phi_{\pi(k)\ell} \circ \gamma\mid_{S_k}.
\]
The \ta\ \emph{respects $\Gamma_i$} if it respects every $\gamma \in \Gamma_i$.

Let $k\in I_i$ and let $\ell,\ell'\in I_j$. 
We call the bonding maps from $k$ to $\ell$ and $\ell'$ \emph{consistent} if there exists $\gamma \in \Gamma_j$ such that 
\[
    \phi_{k\ell} = \gamma \circ \phi_{k\ell'}.
\]
The bonding maps between $J_i \subseteq I_i$ and $J_j\subseteq I_j$ are \emph{consistent} if they are consistent for all $k \in J_i$ and $\ell,\ell' \in J_j$.

The \ta\ $G_1\ast G_2$ is of \emph{Type 1 respecting the actions of~$\Gamma_1$ and $\Gamma_2$} if the following holds:
\begin{enumerate}[(i)]
\item \label{itm:T1respaction} The \ta\ respects $\Gamma_1$ and $\Gamma_2$.
\item \label{itm:T1esim} The bonding maps between $I_1$ and $I_2$ are consistent.
\end{enumerate}

The \ta\ $G_1\ast G_2$ is of \emph{Type 2 respecting the actions of~$\Gamma_1$ and $\Gamma_2$} if the following holds:
\begin{enumerate}[(i)]
\item[(o)] \label{itm:T2prelim} $G_1=G_2=:G$, $\Gamma_1=\Gamma_2=:\Gamma$ and  $I_1=I_2=:I$,\footnote{Formally we would have to work with a bijective map $I_1\to I_2$ since we asked $I_1$ and $I_2$ to be disjoint.} and there exists $J \subseteq I$ such that $f(\fwd e) \in J$ if and only if $f(\bwd e) \notin J$.
\item \label{itm:T2respaction} The \ta\ respects $\Gamma$.
\item \label{itm:T2esim} The bonding maps between $J$ and $I\setminus J$ are consistent.
\end{enumerate}

The \ta\ $G_1\ast G_2$ \emph{respects the actions (of $\Gamma_1$ and $\Gamma_2$)} if it is of either Type~1 or Type~2 respecting the actions $\Gamma_1$ and $\Gamma_2$.

\section{Asymptotic dimension}

In this section, we state the definitions and cite the results that we need regarding the asymptotic dimension.
We keep this section in the general setting of metric spaces instead of restricting it to graphs since the cited results are all for metric spaces.

Let $X$ be a metric space.
A cover $\Vcal$ of~$X$ is \emph{uniformly bounded} if $\sup\{\diam(V) \mid V\in\Vcal\}$ is finite.
The \emph{multiplicity} $\mult(\Vcal)$ of $\Vcal$ of~$X$ is the largest number of elements of~$\Vcal$ that contain a common point of~$X$ or, equivalently, it is the smallest number $n$ such that every $x\in X$ belongs to at most $n$ elements of~$\Vcal$.
A cover $\Ucal$ \emph{refines} $\Vcal$ if for every $U\in\Ucal$ there is a $V\in\Vcal$ with $U\sub V$.

The space $X$ has \emph{asymptotic dimension at most $n$} if for every uniformly bounded open cover $\Vcal$ of~$X$ there is a uniformly bounded open cover $\Ucal$ of~$X$ of multiplicity at most $n+1$ so that $\Vcal$ refines $\Ucal$.
It has \emph{asymptotic dimension} $n$ and we write $\asdim(X)=n$ if it has asymptotic dimension at most $n$ but not at most $n-1$.
A family $\Vcal$ of subsets of a metric space~$X$ is \emph{$r$-disjoint} for $r>0$ if $d(V,V')\geq r$ for all $V\neq V'\in\Vcal$.

\begin{prop}\cite[Theorem 1]{BD-AsdimBedlewo}
Let $X$ be a metric space.
Then $\asdim(X)\leq n$ if and only if for every $r>0$ there exist $r$-disjoint uniformly bounded families $\Vcal_0,\ldots,\Vcal_n$ of subsets of~$X$ such that $\bigcup_{0\leq i\leq n}\Vcal_i$ is a cover of~$X$.\qed
\end{prop}

Let$Y$ be a metric space.
A map $f\colon X\to Y$ is a \emph{coarse equivalence} if there are non-decreasing unbounded functions $\varrho_1,\varrho_2\colon \real_+\cup\{\infty\}\to\real_+\cup\{\infty\}$ such that
\[
\varrho_1(d_X(x,x'))\leq d_Y(f(x),f(x'))\leq \varrho_2(d_X(x,x'))
\]
for all $x,x'\in X$.

By~\cite[Proposition 22]{BD-AsdimGroups}, the asymptotic dimension is invariant under coarse equivalence.
As quasi-isometries are coarse equivalences, we directly have the following.

\begin{prop}\label{prop_AsdimAndQI}
The asymptotic dimension is a \qiy\ invariance.\qed
\end{prop}

The next lemma says that restricting ourselves to subspaces does not increase the asymptotic dimension.

\begin{lem}\label{lem_AsdimSubset}\cite[Proposition 23]{BD-Asdim}
Let $X$ be a metric space and $Y\sub X$.
Then $\asdim(Y)\leq \asdim(X)$.\qed
\end{lem}

A family $(X_i)_{i\in I}$ of subsets of~$X$ satisfies the inequality
$\asdim(X_i) \leq n$ \emph{uniformly} if for every $r > 0$ there exists a $R\in\N$ such that for every $i\in I$ there exist $r$-disjoint families $\Vcal_i^0,\ldots,\Vcal_i^n$ of $R$-bounded subsets of~$X_i$ such that $\bigcup_{0\leq j\leq n}\Vcal_i^j$ is a cover of~$X_i$.

Theorem~\ref{thm_InfinUnionThm} is the Infinite Union Theorem for the asymptotic dimension.

\begin{thm}\label{thm_InfinUnionThm}\cite[Theorem~1]{BD-AsdimGroups}
Let $X$ be a metric space and let $n\in\N$ such that $X = \bigcup_{i\in I} X_i$ for some family $(X_i)_{i\in I}$ with $\asdim(X_i) \leq n$ uniformly for $(X_i)_{i\in I}$.
For every $r > 0$, let $Y_r \sub X$ with $\asdim(Y_r) \leq n$ such that $d(X_i \sm Y_r,X_j \sm Y_r) \geq r$ for all $X_i \neq X_j$.
Then $\asdim(X) \leq n$.\qed
\end{thm}

A consequence of Theorem~\ref{thm_InfinUnionThm} is the Finite Union Theorem, Theorem~\ref{thm_FinUnionThm}.

\begin{thm}\label{thm_FinUnionThm}\cite[Finite Union Theorem]{BD-AsdimGroups}
Let $X=A\cup B$ be a metric space.
Then $\asdim(X) \leq \max\{\asdim(A),\asdim(B)\}$.\qed
\end{thm}

Let $\Ucal$ be a covering of a metric space~$X$.
We denote by
\[
L(\Ucal) := \inf_{U\in\Ucal}\{\sup_{x\in X}\{d(x,X\sm U)\}\}
\]
the \emph{Lebesgue number} of~$\Ucal$.

Let $r>0$, $d>0$ and $n\in\N$.
We write $(r, d) - \dim(X) \leq n$ if there exists a $d$-bounded cover $\Vcal$ of $X$ with $\mult(\Vcal)\leq n + 1$ and with $L(\Vcal) > r$.
We call such a cover an \emph{$(r, d)$-cover} of~$X$.

\begin{prop}\label{prop_AsdimRDCover}\cite[Proposition 2.1]{D-AsdimAmalgamCoxeter}
Let $X$ be a metric space.
Then $X$ has symptotic dimension at most~$n$ if and only if there exists a function
$d\colon\real_{>0}\to\real_{>0}$ such that $(r,d(r))-\dim(X)\leq  n$ for all $r > 0$.\qed
\end{prop}

For a subset $A$ of~$X$ we denote by $\rand A$ its boundary and by $\Int(A)$ its interior, i.\,e.\ $A\sm\rand A$.
A \emph{partition} of a metric space $X$ is a presentation as a union $X=\bigcup_{i\in I} W_i$ such that $Int(W_i)\cap Int(W_j)=\es$ for all $i\neq j$.

The last result that we need for our proof of Theorem~\ref{thm_main} is the Partition Theorem, Theorem~\ref{thm_PartThm}.

\begin{thm}\label{thm_PartThm}\cite[Theorem 2.10]{D-AsdimAmalgamCoxeter} Let $X$ be a geodesic metric space and let $n\in\N$.
If for every $R > 0$ there exists $d > 0$ and a cover $\Vcal$ of~$X$ with $\Int(U)\cap \Int(V)=\es$ for all $U,V\in\Vcal$, with $\asdim(V) \leq n$ uniformly for~$\Vcal$ and such that $(R, d) - \dim(\bigcup_\Vcal\rand V) \leq n - 1$, where $\rand V$ is taken with the metric restricted from~$X$, then $\asdim(X)\leq n$.\qed
\end{thm}

\section{Proof of Theorem~\ref{thm_main}}

In this section, we will prove our main result, Theorem~\ref{thm_main}.
Before we do that, we need some notations and we prove a lemma.

Let $G_1\ast G_2$ be a \ta\ of locally finite \qt\ connected graphs over the connecting tree~$T$.
Let $t\in V(T)$.
For every $m\in\N$, let $O_m$ be the subgraph of $G+H$ induced by the vertex set
\[
\{x\in V(G_1^u), y\in V(G_2^v)\mid d(t,u)=m, d(t,v)=m\}
\]
and let $Q_m$ be the subgraph of $G+H$ induced by the vertex set
\[
\bigcup_{n\leq m} V(O_n).
\]

\begin{lem}\label{lem_T-Prop2.8}
Let $G_1$ and $G_2$ be \qt\ locally finite graphs with asymptotic dimension at most~$n$.
Then the following holds.
\begin{enumerate}[\rm (i)]
\item\label{itm_T-Prop2.8_1} For every $m\in\N$ we have $\asdim(O_m)\leq n$.
\item\label{itm_T-Prop2.8_2} For every $m\in\N$ we have $\asdim(Q_m)\leq n$.
\end{enumerate}
\end{lem}

\begin{proof}
We prove (\ref{itm_T-Prop2.8_1}) inductively.
Since $G_1$ and $G_2$ have asymptotic dimension at most~$n$, it follows from the definition that each $G_i^u$ with the metric inherited by $G_1+G_2$ has asymptotic dimension at most~$n$ and that the family of all $G_i^u$ satisfies $\asdim(G_i^u)$ uniformly.

Since $O_0=G_i^t$, we have $\asdim(G_i^{t})\leq n$.
We are going to apply Theorem~\ref{thm_InfinUnionThm}.
Let $r>0$ and let $U_m$ be the graph induced by the vertices in~$O_m$ of distance at most $r$ to $O_{m-1}$.
Since $G_i$ is \qt\ and there is a unique adhesion set in each $G_j^u$ with $d(t,u)=m-1$ that is not adjacent to~$O_m$, we conclude that $O_{m-1}$ is either empty or \qi\ to~$U_m$.
In particular, we have $\asdim(U_m)\leq n$ by induction and by Proposition~\ref{prop_AsdimAndQI}.

Let $u,v$ be distinct vertices of~$T$ of distance $m$ to~$t$.
Any path $P$ from $V(G_j^u)\sm V(U_m)$ to $V(G_j^v)\sm V(U_m)$ must pass through $G_{3-j}^w$, where $w$ is the neighbour of~$u$ on the unique $t$-$u$ path in~$T$.
Thus, $P$ has length at least $d(V(G_j^u)\sm V(U_m), V(G_{3-j}^w))>r$ and hence Theorem~\ref{thm_InfinUnionThm} implies $\asdim(O_m)\leq n$.

By Theorem~\ref{thm_FinUnionThm}, $\asdim(Q_m)\leq n$ follows directly from $\asdim(O_m)\leq n$.
\end{proof}

Now we are ready to prove our main theorem.

\begin{proof}[Proof of Theorem~\ref{thm_main}]
By Remark~\ref{rem_QI+to*} and Proposition~\ref{prop_AsdimAndQI} it suffices to prove the assertion for $H:=G_1+G_2$ instead of~$G$.
Let $T$ be the connecting tree of our \ta.
Let $t_1t_2\in E(T)$ such that the graph associated to $t_i$ is~$G_i$.
Let $\Scal_i$ be a set of representatives of the orbits of adhesion sets in~$G_i^{t_i}$ under the action of $\Gamma_i^{t_i}$ and let $\Scal_i^u$ be its image in~$G_i^u$.
Let
\[
n:=\max\{\asdim(G_1),\asdim(G_2),\asdim(S)+1\mid S\in \Scal_1\cup\Scal_2\}.
\]
Let $\pi_T\colon V(H)\to V(T)$ be the canonical map that maps $x\in G_i^u$ to~$u$.
For $t\in V(T)$ let $T^t$ be the subtree induced by all $t'\in V(T)$ that are separated by~$t$ from~$t_1$.
Let $r,R\in\N$ such that $r>4R$ and $r$ is even.

Let $U_r$ be the subgraph of~$H$ induced by
\begin{align*}
&\left(\pi_T\inv(B_{r-1}(t_1))\cap \{v\in V(H)\mid d(v,\bigcup\Scal_1^{t_1})\geq R\}\right)\\
\cup&\bigcup\left\{ \left(B_r\left(\bigcup\Scal_j^v\right)\cap G_j^v\right)\mid d(v,t_1)=r\right\}
\end{align*}
and set
\[
M_R:=\{v\in V(H)\mid d(v,\bigcup\Scal_1^{t_1})=R\}.
\]
Let $W_r$ be $U_r$ without those edges that have both their incident vertices in~$M_R$.

Let us extend the definitions of $W_r$ and $M_r$ to all vertices $t$ of~$T$ of even distance to~$t_1$: we set
\[
W_r^t:=f_t(W_r)\cap H[\pi_T\inv(V(T^t))]
\]
and
\[
M_R^t:=f_t(M_R)\cap\pi_T\inv(V(T^t)),
\]
where $H[\pi_T\inv(V(T^t))]$ is the subgraph of~$H$ induced by $\pi_T\inv(V(T^t))$ and where $f_t\in\Aut(H)$ maps $t_1$ to~$t$ such that the adhesion set separating $G_1^t$ from $G_1^{t_1}$ lies in $f_t(\Scal_1)$.
Note that this definition does not depend on the particular choice of~$f_t$.

We consider the set
\[
\Wcal:=\{W_r^t\mid d(t,t_1)\in r\N\}\cup\{W^0\},
\]
where $W^0$ is the graph induced by $B_R(\bigcup\Scal_1^{t_1})$.
Our aim is to show that we can apply Theorem~\ref{thm_PartThm} for the set~$\Wcal$.
It follows directly from its construction that $\Wcal$ is a cover of~$H$.
The elements of~$\Wcal$ are edge-disjoint by construction but they may share vertices that lie in $M_R$ or its images~$M_R^t$.
Thus, $\Int(U)\cap\Int(W)=\es$ for all $U,W\in\Wcal$.

Since $W^0$ is \qi\ to~$\bigcup\Scal_1^{t_1}$, their asymptotic dimensions coincide.
By Theorem~\ref{thm_FinUnionThm}, we conclude that it is at most~$n-1$.
Thus, in order to show that $\asdim(W)\leq n-1$ uniformly for $\Wcal$, it suffices to show it for $\Wcal\sm\{W^0\}$.
But this follows from Lemmas~\ref{lem_AsdimSubset} and~\ref{lem_T-Prop2.8} since $W_r\sub Q_r$ and $W_r^t$ is isomorphic to a subgraph of~$W_r$.

Let us prove $(R,d)-dim(Z)\leq n-1$ for $Z:=\bigcup_{W\in \Wcal}\rand W$.
We have
\[
Z=\bigcup\left\{\rand M_R^t\mid t\in V(T), d(t,t_1)\in r\N\right\}.
\]
Note that $M_R$ is \qi\ to $\bigcup\Scal_1$.
Since $\Scal_1$ has at most two elements each of which has asymptotic dimension at most $n-1$, we conclude $\asdim(\bigcup\Scal_1)\leq n-1$ by Theorem~\ref{thm_FinUnionThm}.
Thus, we have $\asdim(M_R)\leq n-1$.
By Proposition~\ref{prop_AsdimRDCover}, there exists a $d>0$ and an $(R,d)$-cover $\Ucal$ of~$M_R$ with multiplicity at most~$n$.

Set $\Vcal:=\Ucal\cup\bigcup\{g_t(U)\cap M_R^t\mid U\in\Ucal, d(t,t_1)\in r\N_{>0}\}$, where $g_t$ is an automorphism of~$H$ that maps $G_1^{t_1}$ to $G_1^t$.
Obviously, $\Vcal$ is a $d$-bounded cover of~$Z$.
Since the sets $M_R^t$ are pairwise disjoint, the multiplicity of~$\Vcal$ is at most $n-1$.

To prove $(R,d)-dim(Z)\leq n-1$, it remains to prove $L(\Vcal)>R$.
For that, we just have to show $d(M_R^t,M_R^{t'})>R$ for all $t,t'\in V(T)$ with $d(t,t_1),d(t',t_1)\in r\N$.
Let $v$ be the vertex on the $t$-$t'$ path in $T$ closest to~$t_1$.
If $v$ is distinct from $t$ and~$t'$, then every path with one end vertex in $M_R^t$ and the other in $M_R^{t'}$ must pass through $G_j^v$.
So we have $d(M_R^{t'},M_R^t)> d(G_j^v,M_R^t)>R$.
If $v$ is not distinct from $t$ and $t'$, the let us assume $v=t$.
Let $S_t$, $S_{t'}$ be the adhesion set in $G_1^t$, in $G_1^{t'}$ separating $G_1^{t_1}$ from $G_1^t$, from $G_1^{t'}$, respectively.
Let $w$ be the neighbour of~$t$ on the unique $t$-$t'$ path and let $S_v$ be the adhesion set in $G_2^v$ separating $G_2^v$ from $G_1^t$.
Then $d(M_R^t,M_R^{t'})\geq d(S_v,S_t')\geq 3R>R$.
So we have $d(M_R^t,M_R^{t'})>R$ in both cases.

Now Theorem~\ref{thm_PartThm} implies $\asdim(H)\leq n$ and thus we have $\asdim(G)\leq n$ by Remark~\ref{rem_QI+to*} and Proposition~\ref{prop_AsdimAndQI}.
\end{proof}

\bibliographystyle{amsplain}
\bibliography{D:/Repository/Bibs/Bibs}

\end{document}